\documentclass[12pt]{amsart}

\newtheorem{theorem}{Theorem}
\newtheorem{lemma}{Lemma}
\newtheorem{prop}{Proposition}

\theoremstyle{definition}

\theoremstyle{remark}
\newtheorem{remark}{Remark}
\numberwithin{equation}{section}

\usepackage{a4wide,amsmath,amssymb}
\usepackage{color}

\def\re{\mathbb{R}}
\def\N{\mathbb{N}}

\def\({\left(}
\def\){\right)}

\def\la{\lambda}

\def\intO{\int_{\Omega}}

\begin{document}
\title[]{On a weighted Trudinger-Moser type inequality on the whole space and related maximizing problem}
\author{Van Hoang Nguyen}
\address{Institut de Math\'ematiques de Toulouse, Universit\'e Paul Sabatier 118 route de Narbonne, 31062 Toulouse C\'edex 09, France}
\email{van-hoang.nguyen@math.univ-toulouse.fr}

\author{Futoshi Takahashi}
\address{Department of Mathematics, Osaka City University \& OCAMI, Sumiyoshi-ku, Osaka, 558-8585, Japan}

\email{futoshi@sci.osaka-cu.ac.jp}

\subjclass[2010]{Primary 35A23; Secondary 26D10.}

\keywords{Trudinger-Moser inequality, weighted Sobolev spaces, maximizing problem.}
\date{\today}

\dedicatory{}

\begin{abstract}
In this paper, we establish a weighted Trudinger-Moser type inequality with the full Sobolev norm constraint on the whole Euclidean space.
Main tool is the singular Trudinger-Moser inequality on the whole space recently established by Adimurthi and Yang, and a transformation of functions.
We also discuss the existence and non-existence of maximizers for the associated variational problem.
\end{abstract}

\maketitle

%
%
\section{Introduction}

Let $\Omega \subset \re^N$, $N \ge 2$ be a domain with finite volume.
Then the Sobolev embedding theorem assures that $W^{1,N}_0(\Omega) \hookrightarrow L^q(\Omega)$ for any $q \in [1, +\infty)$,
however, as the function $\log \( \log (e/|x|) \) \in W^{1,N}_0(B)$, $B$ the unit ball in $\re^N$, shows,
the embedding $W^{1,N}_0(\Omega) \hookrightarrow L^{\infty}(\Omega)$ does not hold.
Instead, functions in $W^{1,N}_0(\Omega)$ enjoy the exponential summability:
\[
	W^{1,N}_0(\Omega) \hookrightarrow \{ u \in L^N(\Omega) \, : \, \intO \exp \(\alpha |u|^{\frac{N}{N-1}} \) dx < \infty \quad \text{for any} \, \alpha > 0 \},
\]
see Yudovich \cite{Yudovich}, Pohozaev \cite{Pohozaev}, and Trudinger \cite{Trudinger}.
Moser \cite{Moser} improved the above embedding as follows, now known as the Trudinger-Moser inequality:
Define
\[
	TM(N, \Omega, \alpha) = \sup_{u \in W^{1,N}_0(\Omega) \atop \| \nabla u \|_{L^N(\Omega)} \le 1} \frac{1}{|\Omega|} \intO \exp (\alpha |u|^{\frac{N}{N-1}}) dx.
\]
Then we have
\begin{align*}
	TM(N, \Omega, \alpha) 
	\begin{cases}
	&< \infty, \quad \alpha \le \alpha_N, \\
	&= \infty, \quad \alpha > \alpha_N,
	\end{cases}
\end{align*}
here and henceforth $\alpha_N = N \omega_{N-1}^{\frac{1}{N-1}}$ and $\omega_{N-1}$ denotes the area of the unit sphere $S^{N-1}$ in $\re^N$. 
On the attainability of the supremum, Carleson-Chang \cite{Carleson-Chang}, Flucher \cite{Flucher}, and Lin \cite{KCLin} proved that 
$TM(N, \Omega, \alpha)$ is attained on any bounded domain for all $0 < \alpha \le \alpha_N$.

Later, Adimurthi-Sandeep \cite{Adimurthi-Sandeep} established a weighted (singular) Trudinger-Moser inequality as follows:
Let $0 \le \beta < N$ and put $\alpha_{N, \beta} = \(\frac{N-\beta}{N}\) \alpha_N$.
Define
\[
	\widetilde{TM}(N, \Omega, \alpha, \beta) = \sup_{u \in W^{1,N}_0(\Omega) \atop \| \nabla u \|_{L^N(\Omega)} \le 1} \frac{1}{|\Omega|} \intO \exp (\alpha |u|^{\frac{N}{N-1}}) \frac{dx}{|x|^{\beta}}.
\]
Then it is proved that
\[
	\widetilde{TM}(N, \Omega, \alpha, \beta) 
	\begin{cases}
	&< \infty, \quad \alpha \le \alpha_{N, \beta}, \\
	&= \infty, \quad \alpha > \alpha_{N, \beta}.
	\end{cases}
\]
On the attainability of the supremum, recently Csat\'o-Roy \cite{Csato-Roy(CVPDE)}, \cite{Csato-Roy(CPDE)} proved that 
$\widetilde{TM}(2, \Omega, \alpha, \beta)$ is attained for $0 < \alpha \le \alpha_{2,\beta} = 2\pi(2-\beta)$ for any bounded domain $\Omega \subset \re^2$.
For other types of weighted Trudinger-Moser inequalities, see for example,
\cite{Calanchi}, \cite{Calanchi-Ruf(JDE)}, \cite{Calanchi-Ruf(NA)}, \cite{Furtado-Medeiros-Severo}, \cite{Lam-Lu}, \cite{Souza}, \cite{Souza-O}, \cite{Yang}, 
to name a few.

On domains with infinite volume, for example on the whole space $\re^N$, the Trudinger-Moser inequality does not hold as it is.
However, several variants are known on the whole space.
In the following, let
\[
	\Phi_N(t) = e^t - \sum_{j=0}^{N-2} \frac{t^j}{j!} 
\]
denote the truncated exponential function.

First,
Ogawa \cite{Ogawa}, Ogawa-Ozawa \cite{Ogawa-Ozawa}, Cao \cite{Cao}, Ozawa \cite{Ozawa(JFA)}, and Adachi-Tanaka \cite{Adachi-Tanaka}
proved that the following inequality holds true, which we call Adachi-Tanaka type Trudinger-Moser inequality:
Define
\begin{align}
\label{AT-sup}
	A(N, \alpha) = \sup_{u \in W^{1,N}(\re^N) \setminus \{ 0 \} \atop \| \nabla u \|_{L^N(\re^N)} \le 1} \frac{1}{\| u \|^N_{L^N(\re^N)}} \int_{\re^N} \Phi_N (\alpha |u|^{\frac{N}{N-1}}) dx.
\end{align}
Then
\begin{equation}
\label{AT-TM}
	A(N, \alpha)
	\begin{cases}
	&< \infty, \quad \alpha \, < \, \alpha_N, \\
	&= \infty, \quad \alpha \ge \alpha_N.
	\end{cases}
\end{equation}
The functional in (\ref{AT-sup})
\[
	F(u) = \frac{1}{\| u \|^N_{L^N(\re^N)}} \int_{\re^N} \Phi_N (\alpha |u|^{\frac{N}{N-1}}) dx
\]
enjoys the scale invariance under the scaling $u(x) \mapsto u_{\la}(x) = u(\la x)$ for $\la > 0$, i.e., $F(u_{\la}) = F(u)$ for any $u \in W^{1,N}(\re^N) \setminus \{ 0 \}$.
Note that the critical exponent $\alpha = \alpha_N$ is not allowed for the finiteness of the supremum.
On the attainability of the supremum, Ishiwata-Nakamura-Wadade \cite{Ishiwata-Nakamura-Wadade} proved that $A(N, \alpha)$ is attained for any $\alpha \in (0, \alpha_N)$.
In this sense, Adachi-Tanaka type Trudinger-Moser inequality has a subcritical nature of the problem.

On the other hand, Ruf \cite{Ruf} and Li-Ruf \cite{Li-Ruf} proved that the following inequality holds true:
Define
\begin{align}
\label{LR-sup}
	B(N, \alpha) = \sup_{u \in W^{1,N}(\re^N) \atop \| u \|_{W^{1,N}(\re^N)} \le 1} \int_{\re^N} \Phi_N (\alpha |u|^{\frac{N}{N-1}}) dx.
\end{align}
Then 
\begin{equation}
\label{LR-TM}
	B(N, \alpha)
	\begin{cases}
	&< \infty, \quad \alpha \, \le \, \alpha_N, \\
	&= \infty, \quad \alpha > \alpha_N.
	\end{cases}
\end{equation}
Here $\| u \|_{W^{1,N}(\re^N)} = \( \| \nabla u \|_{L^N(\re^N)}^N + \| u \|_{L^N(\re^N)}^N \)^{1/N}$ is the full Sobolev norm.
Note that the scale invariance $(u \mapsto u_{\la})$ does not hold for this inequality.
Also the critical exponent $\alpha = \alpha_N$ is permitted to the finiteness of (\ref{LR-sup}). 
Concerning the attainability of $B(N, \alpha)$, it is known that $B(N, \alpha)$ is attained for $0 < \alpha \le \alpha_N$ if $N \ge 3$ \cite{Ruf}.
On the other hand when $N = 2$, there exists an explicit constant $\alpha_* > 0$ related to the Gagliardo-Nirenberg inequality in $\re^2$ 
such that $B(2, \alpha)$ is attained for $\alpha_* < \alpha \le \alpha_2 (= 4\pi)$ \cite{Ruf}, \cite{Ishiwata}.
However, if $\alpha >0$ is sufficiently small, then $B(2,\alpha)$ is not attained \cite{Ishiwata}.
The non-attainability of $B(2,\alpha)$ for $\alpha$ sufficiently small is attributed to the non-compactness of ``vanishing" maximizing sequences, as described in \cite{Ishiwata}.

In the following, we are interested in the weighted version of the Trudinger-Moser inequalities on the whole space.
Let $N \ge 2$, $-\infty < \gamma  < N$ and define the weighted Sobolev space $X^{1,N}_{\gamma}(\re^N)$ as
\begin{align*}
	&X^{1,N}_{\gamma}(\re^N) = {\dot W}^{1,N}(\re^N) \cap L^N(\re^N, |x|^{-\gamma} dx) \\
	&=\{ u \in L^1_{loc}(\re^N) \, : \, \| u \|_{X^{1,N}_{\gamma}(\re^N)} = \( \| \nabla u \|_N^N + \| u \|_{N, \gamma}^N \)^{1/N} < \infty \},
\end{align*}
where we use the notation $\| u \|_{N, \gamma}$ for $\( \int_{\re^N} \frac{|u|^N}{|x|^{\gamma}} dx \)^{1/N}$.
We also denote by $X^{1,N}_{\gamma,rad}(\mathbb R^N)$ the subspace of $X^{1,N}_\gamma(\mathbb R^N)$ consisting of radial functions.
We note that a special form of the Caffarelli-Kohn-Nirenberg inequality in \cite{Caffarelli-Kohn-Nirenberg}: 
\begin{equation}
\label{CKN}
	\| u \|_{N, \beta} \le C \| u \|_{N, \gamma}^{\frac{N-\beta}{N-\gamma}} \| \nabla u \|_{N}^{1 - \frac{N-\beta}{N-\gamma}}
\end{equation}
implies that $X^{1,N}_{\gamma}(\re^N) \subset  X^{1,N}_{\beta}(\re^N)$ when $\gamma \le \beta$.
From now on, we assume
\begin{equation}
\label{assumption:weighted AT}
	N \ge 2, \quad -\infty < \gamma \le \beta < N
\end{equation}
and put $\alpha_{N, \beta} = \(\frac{N-\beta}{N}\) \alpha_N$.

Recently, Ishiwata-Nakamura-Wadade \cite{Ishiwata-Nakamura-Wadade} proved that the following weighted Adachi-Tanaka type Trudinger-Moser inequality holds true:
Define
\begin{equation}
\label{weighted AT-sup(radial)}
	\tilde{A}_{rad}(N, \alpha, \beta, \gamma) 
	= \sup_{u \in X^{1,N}_{\gamma, rad}(\re^N) \setminus \{ 0 \} \atop \| \nabla u \|_{L^N(\re^N)} \le 1} \frac{1}{\| u \|_{N, \gamma}^{N(\frac{N-\beta}{N-\gamma})}} 
	\int_{\re^N} \Phi_N (\alpha |u|^{\frac{N}{N-1}}) \frac{dx}{|x|^{\beta}}.
\end{equation}
Then for $N, \beta, \gamma$ satisfying (\ref{assumption:weighted AT}), we have
\begin{align}
\label{weighted AT(radial)}
	\tilde{A}_{rad}(N, \alpha, \beta, \gamma) 
	&\begin{cases}
	&< \infty, \quad \alpha \, < \, \alpha_{N, \beta}, \\
	&= \infty, \quad \alpha \ge \alpha_{N, \beta}.
	\end{cases}
\end{align}
Later,  Dong-Lu \cite{Dong-Lu} extends the result in the non-radial setting.
Let
\begin{equation}
\label{weighted AT-sup}
	\tilde{A}(N, \alpha, \beta, \gamma) 
	= \sup_{u \in X^{1,N}_{\gamma}(\re^N) \setminus \{ 0 \} \atop \| \nabla u \|_{L^N(\re^N)} \le 1} \frac{1}{\| u \|_{N, \gamma}^{N(\frac{N-\beta}{N-\gamma})}} 
	\int_{\re^N} \Phi_N (\alpha |u|^{\frac{N}{N-1}}) \frac{dx}{|x|^{\beta}}.
\end{equation}
Then the corresponding result holds true also for $\tilde{A}(N, \alpha, \beta, \gamma)$.
Attainability of the best constant (\ref{weighted AT-sup(radial)}), (\ref{weighted AT-sup}) is also considered in \cite{Ishiwata-Nakamura-Wadade} and \cite{Dong-Lu}:
$\tilde{A}_{rad}(N, \alpha, \beta, \gamma)$ and $\tilde{A}(N, \alpha, \beta, \gamma)$ are attained for any $0 < \alpha < \alpha_{N, \beta}$.

First purpose of this note is to establish the weighted Li-Ruf type Trudinger-Moser inequality on the weighted Sobolev space $X^{1,N}_{\gamma}(\re^N)$
with $N, \beta, \gamma$ satisfying (\ref{assumption:weighted AT}).
Define
\begin{align}
\label{weighted LR-sup(radial)}
	&\tilde{B}_{rad}(N, \alpha, \beta, \gamma)
	= \sup_{u \in X^{1,N}_{\gamma, rad}(\re^N) \atop \| u \|_{X^{1,N}_{\gamma}(\re^N)} \le 1} \int_{\re^N} \Phi_N (\alpha |u|^{\frac{N}{N-1}}) \frac{dx}{|x|^{\beta}}, \\
\label{weighted LR-sup}
	&\tilde{B}(N, \alpha, \beta, \gamma)
	= \sup_{u \in X^{1,N}_{\gamma}(\re^N) \atop \| u \|_{X^{1,N}_{\gamma}(\re^N)} \le 1} \int_{\re^N} \Phi_N (\alpha |u|^{\frac{N}{N-1}}) \frac{dx}{|x|^{\beta}}.
\end{align}

\begin{theorem}(Weighted Li-Ruf type inequality)
\label{Theorem:weighted LR}
Assume (\ref{assumption:weighted AT}) and put $\alpha_{N, \beta} = \(\frac{N-\beta}{N}\) \alpha_N$. 
Then we have
\begin{align}
\label{weighted LR(radial)}
	&\tilde{B}_{rad}(N, \alpha, \beta, \gamma)
	\begin{cases}
	&< \infty, \quad \alpha \, \le \, \alpha_{N, \beta}, \\
	&= \infty, \quad \alpha > \alpha_{N, \beta}.
	\end{cases}
\end{align}
Furthermore if $0 \le \gamma \le \beta < N$, we have
\begin{align}
\label{weighted LR}
	&\tilde{B}(N, \alpha, \beta, \gamma)
	\begin{cases}
	&< \infty, \quad \alpha \, \le \, \alpha_{N, \beta}, \\
	&= \infty, \quad \alpha > \alpha_{N, \beta}.
	\end{cases}
\end{align}
\end{theorem}

We also study the existence and non-existence of maximizers for the weighted Trudinger-Moser inequalities \eqref{weighted LR(radial)} and \eqref{weighted LR}.

\begin{theorem}
\label{Maximizers(radial)}
Assume (\ref{assumption:weighted AT}). Then the following statements hold.
\begin{enumerate}
\item[(i)] If $N \geq 3$ then $\tilde{B}_{rad}(N,\alpha,\beta,\gamma)$ is attained for any $0< \alpha \leq \alpha_{N,\beta}$.
\item[(ii)] If $N=2$ then $\tilde{B}_{rad}(2,\alpha,\beta,\gamma)$ is attained for any $0< \alpha \leq \alpha_{2,\beta}$ if $\beta > \gamma$, 
while there exists $\alpha_* > 0$ such that $\tilde{B}_{rad}(2,\alpha,\beta,\beta)$ is attained for any $\alpha_* < \alpha < \alpha_{2,\beta}$.
\item[(iii)] $\tilde{B}_{rad}(2,\alpha,\beta,\beta)$ is not attained for sufficiently small $\alpha > 0$.
\end{enumerate}
\end{theorem}

\begin{theorem}
\label{Maximizers}
Let $N\geq 2$, $0\leq \gamma \leq \beta < N$. 
Then the following statements hold.
\begin{enumerate}
\item[(i)] If $N \geq 3$ then $\tilde{B}(N,\alpha,\beta,\gamma)$ is attained for any $0< \alpha \leq \alpha_{N,\beta}$.
\item[(ii)] If $N = 2$ then $\tilde{B}(2,\alpha,\beta,\gamma)$ is attained for any $0< \alpha \leq \alpha_{2,\beta}$ if $\beta > \gamma$, 
while there exists $\alpha_* > 0$ such that $\tilde{B}(2,\alpha,\beta,\beta)$ is attained for any $\alpha_* < \alpha < \alpha_{2,\beta}$.
\item[(iii)] $\tilde{B}(2,\alpha,\beta,\beta)$ is not attained for sufficiently small $\alpha > 0$.
\end{enumerate}
\end{theorem}

Next, we study the relation between the suprema of Adachi-Tanaka type and Li-Ruf type weighted Trudinger-Moser inequalities,
along the line of Lam-Lu-Zhang \cite{Lam-Lu-Zhang}.
Set $\tilde{B}_{rad}(N, \beta, \gamma) = \tilde{B}_{rad}(N, \alpha_{N, \beta},\beta, \gamma)$ in (\ref{weighted LR-sup(radial)}), 
and 
$\tilde{B}(N, \beta, \gamma) = \tilde{B}(N, \alpha_{N, \beta},\beta, \gamma)$ in (\ref{weighted LR-sup}), 
i.e.,
\begin{align}
\label{weighted LR-sup-critical(radial)}
	&\tilde{B}_{rad}(N, \beta, \gamma)
	= \sup_{u \in X^{1,N}_{\gamma, rad}(\re^N) \atop \| u \|_{X^{1,N}_{\gamma}} \le 1} \int_{\re^N} \Phi_N (\alpha_{N, \beta} |u|^{\frac{N}{N-1}}) \frac{dx}{|x|^{\beta}}, \\
\label{weighted LR-sup-critical}
	&\tilde{B}(N, \beta, \gamma)
	= \sup_{u \in X^{1,N}_{\gamma}(\re^N) \atop \| u \|_{X^{1,N}_{\gamma}} \le 1} \int_{\re^N} \Phi_N (\alpha_{N, \beta} |u|^{\frac{N}{N-1}}) \frac{dx}{|x|^{\beta}},
\end{align}
for $N, \beta, \gamma$ satisfying (\ref{assumption:weighted AT}).
Then $\tilde{B}_{rad}(N, \beta, \gamma) < \infty$, and $\tilde{B}(N, \beta, \gamma) < \infty$ if $\gamma \ge 0$, by Theorem \ref{Theorem:weighted LR}.

\begin{theorem}(Relation)
Assume (\ref{assumption:weighted AT}).
Then we have
\label{Theorem:relation}
\[
	\tilde{B}_{rad}(N, \beta, \gamma) = \sup_{\alpha \in (0, \alpha_{N,\beta})} \( \frac{1 - \(\frac{\alpha}{\alpha_{N,\beta}}\)^{N-1}}{\(\frac{\alpha}{\alpha_{N,\beta}}\)^{N-1}} \)^{\frac{N-\beta}{N-\gamma}} 
\tilde{A}_{rad}(N, \alpha, \beta, \gamma).
\]
Furthermore if $\gamma \ge 0$, we have
\[
	\tilde{B}(N, \beta, \gamma) = \sup_{\alpha \in (0, \alpha_{N,\beta})} \( \frac{1 - \(\frac{\alpha}{\alpha_{N,\beta}}\)^{N-1}}{\(\frac{\alpha}{\alpha_{N,\beta}}\)^{N-1}} \)^{\frac{N-\beta}{N-\gamma}} 
\tilde{A}(N, \alpha, \beta, \gamma).
\]
\end{theorem}
Note that this implies $\tilde{A}_{rad}(N, \alpha, \beta, \gamma) < \infty$ for $N, \beta, \gamma$ satisfying (\ref{assumption:weighted AT}),
and $\tilde{A}(N, \alpha, \beta, \gamma) < \infty$ if $0 \le \gamma \le \beta < N$, by Theorem \ref{Theorem:weighted LR}.

Furthermore, we prove how $\tilde{A}_{rad}(N, \alpha, \beta, \gamma)$ and $\tilde{A}(N, \alpha, \beta, \gamma)$ behaves as $\alpha$ approaches to $\alpha_{N, \beta}$ from the below:
\begin{theorem}(Asymptotic behavior of Adachi-Tanaka supremum)
\label{Theorem:asymptotic}
Assume (\ref{assumption:weighted AT}).
Then there exist positive constants $C_1, C_2$ (depending on $N$, $\beta$, and $\gamma$) such that for $\alpha$ close enough to $\alpha_{N, \beta}$,
the estimate
\[
	\( \frac{C_1}{1 - \(\frac{\alpha}{\alpha_{N,\beta}}\)^{N-1}} \)^{\frac{N-\beta}{N-\gamma}} \le \tilde{A}_{rad}(N, \alpha, \beta, \gamma) 
\le \( \frac{C_2}{1 - \(\frac{\alpha}{\alpha_{N,\beta}}\)^{N-1}} \)^{\frac{N-\beta}{N-\gamma}}
\]
holds.
Corresponding estimates hold true for $\tilde{A}(N, \alpha, \beta, \gamma)$ if $\gamma \ge 0$.
\end{theorem}
Note that the estimate from the above follows from Theorem \ref{Theorem:relation}.
On the other hand, we will see that the estimate from the below follows from a computation using the Moser sequence.

The organization of the paper is as follows:
In section 2, we prove Theorem \ref{Theorem:weighted LR}. 
Main tools are a transformation which relates a function in $X^{1,N}_\gamma(\re^N)$ to a function in $W^{1,N}(\re^N)$, 
and the singular Trudinger-Moser type inequality recently proved by Adimurthi and Yang \cite{Adimurthi-Yang}, see also de Souza and de O \cite{Souza-O}.
In section 3, we prove the existence part of Theorems \ref{Maximizers(radial)}, \ref{Maximizers} (i) (ii).
In section 4, we prove the nonexistence part of Theorem \ref{Maximizers(radial)}, \ref{Maximizers} (iii).
Finally in  section 5, we prove Theorem \ref{Theorem:relation} and Theorem \ref{Theorem:asymptotic}.
The letter $C$ will denote various positive constant which varies from line to line, but is independent of functions under consideration.

%
%

\section{Proof of Theorem \ref{Theorem:weighted LR}.}

In this section, we prove Theorem \ref{Theorem:weighted LR}.
We will use the following singular Trudinger-Moser inequality on the whole space $\mathbb R^N$:
For any $\beta \in [0, N)$, define
\begin{equation}
\label{singular TM-sup}
	\tilde{B}(N,\alpha,\beta,0) = \sup_{u \in W^{1,N}(\mathbb R^N), \atop \|u \|_{W^{1,N}} \leq 1} \int_{\mathbb R^N} \Phi_N(\alpha |u|^{\frac N{N-1}}) \frac{dx}{|x|^{\beta}}.
\end{equation}
Then it holds
\begin{equation}
\label{singular TM}
	\tilde{B}(N, \alpha, \beta, 0)
	\begin{cases}
	&< \infty, \quad \alpha \, \le \, \alpha_{N, \beta}, \\
	&= \infty, \quad \alpha > \alpha_{N, \beta}.
	\end{cases}
\end{equation}
Here $\| u \|_{W^{1,N}} = \( \|\nabla u\|_N^N + \|u\|_N^N \)^{1/N}$ denotes the full norm of the Sobolev space $W^{1,N}(\re^N)$.  
Note that the inequality \eqref{singular TM} was first established by Ruf \cite{Ruf} for the case $N =2$ and $\beta =0$. 
It then was extended to the case $N\geq 3$ and $\beta =0$ by Li and Ruf \cite{Li-Ruf}. 
The case $N\geq 2$ and $\beta \in (0,N)$ was proved by Adimurthi and Yang \cite{Adimurthi-Yang}, see also de Souza and de O \cite{Souza-O}.

\vspace{1em}\noindent
{\it Proof of Theorem \ref{Theorem:weighted LR}}: 
We define the vector-valued function $F$ by 
\[
	F(x) = \left(\frac {N-\gamma}N\right)^{\frac{N}{N-\gamma}}|x|^{\frac{\gamma}{N-\gamma}} x.
\]
Its Jacobian matrix is
\begin{align*}
	DF(x) &= \left(\frac {N-\gamma}N\right)^{\frac{N}{N-\gamma}}|x|^{\frac{\gamma}{N-\gamma}} \left( Id_{N} + \frac{\gamma}{N-\gamma} \frac{x}{|x|} \, \otimes \, \frac{x}{|x|}\right)\\
	&=\frac{N-\gamma}N |F(x)|^{\frac \gamma N}\left( Id_{N} + \frac{\gamma}{N-\gamma} \frac{x}{|x|} \, \otimes \, \frac{x}{|x|}\right). 
\end{align*}
where $Id_N$ denotes the $N\times N$ identity matrix and $v\otimes v = (v_i v_j)_{1 \le i, j \le N}$ denotes the matrix 
corresponding to the orthogonal projection onto the line generated by the unit vector $v = (v_1, \cdots, v_N) \in \re^N$, i.e., the map $x\mapsto (x\cdot v) v$. 
Since a matrix of the form $I + \alpha v \otimes v$, $\alpha \in \re$, has eigenvalues $1$ (with multiplicity $N-1$) and $1 + \alpha$ (with multiplicity $1$),
we see that
\begin{equation}\label{eq:JacobianF}
	det(DF(x)) = \left(\frac{N-\gamma}N\right)^{N-1} |F(x)|^{\gamma}. 
\end{equation}
Let $u \in X^{1,N}_\gamma(\mathbb R^N)$ be such that $\|u\|_{X^{1,N}_\gamma} \leq 1$. We introduce a change of functions as follows.
\begin{equation}\label{eq:changefunct}
	v(x) = \left(\frac{N-\gamma}N\right)^{\frac {N-1}N}\, u(F(x)).
\end{equation}
A simple calculation shows that
\begin{align*}
	\nabla v(x)&= \left(\frac {N-\gamma}N\right)^{\frac{N-1}N}DF(x) (\nabla u(F(x)))\\
	& = \left(\frac{N-\gamma}N\right)^{\frac{2N-1}N}|F(x)|^{\frac{\gamma}{N}} \left(\nabla u(F(x)) + \frac{\gamma}{N-\gamma}\left(\nabla u(F(x)) \cdot \frac x{|x|}\right) \frac x{|x|}\right),
\end{align*}
and hence
\[
	|\nabla v(x)|^2 =\left(\frac{N-\gamma}N\right)^{\frac{2(2N-1)}N}|F(x)|^{\frac{2\gamma}{N}} \left(|\nabla u(F(x))|^2 + \frac{\gamma(2N-\gamma)}{(N-\gamma)^2} \left(\nabla u(F(x)) \cdot \frac x{|x|}\right)^2\right).
\]
Since $\left(\nabla u(F(x)) \cdot \frac x{|x|}\right)^2 \leq |\nabla u(F(x))|^2$, we then have
\begin{equation}
\label{eq:pointwiseineq}
	|\nabla v(x)| \leq \left(\frac {N-\gamma}N\right)^{\frac{N-1}N} |F(x)|^{\frac\gamma{N}} |\nabla u(F(x))|= \left(\det (DF(x))\right)^{\frac1N} |\nabla u(F(x))|
\end{equation}
if $\gamma \ge 0$, with equality if and only if $\left(\nabla u(F(x)) \cdot \frac x{|x|}\right)^2 = |\nabla u(F(x))|^2$ when $\gamma >0$. 
If $\gamma =0$ the inequality \eqref{eq:pointwiseineq} is an equality. 
Note that the inequality \eqref{eq:pointwiseineq} does not hold if $\gamma < 0$ and $u$ is not radial function. 
In fact, a reversed inequality occurs in this case. Moreover, \eqref{eq:pointwiseineq} becomes an equality if $u$ is a radial function for any $-\infty < \gamma < N$. 
Integrating both sides of \eqref{eq:pointwiseineq} on $\mathbb R^N$, we obtain
\begin{equation}
\label{eq:comparenorm}
	\|\nabla v\|_N \leq \|\nabla u\|_N.
\end{equation}
Moreover, for any function $G$ on $[0,\infty)$, using the change of variables, we get
\begin{multline}
\label{eq:changeintegral}
	\int_{\mathbb R^N} G\left(|u(x)|^{\frac N{N-1}}\right) |x|^{-\delta} dx \\ 
	= \left(\frac {N-\gamma}N\right)^{N-1+ \frac{N(\gamma -\delta)}{N-\gamma}}\int_{\mathbb R^N} G\left(\frac N{N-\gamma} |v(y)|^{\frac N{N-1}}\right) |y|^{\frac{N(\gamma -\delta)}{N-\gamma}} dy.
\end{multline}
Consequently, by choosing $G(t) =t^{N-1}$ and $\delta =\gamma$, we get $\|u\|_{N,\gamma} = \|v\|_N$ and hence
\begin{equation}
\label{eq:norm}
	\|u\|_{X^{1,N}_\gamma}^N = \|\nabla u\|_N^N + \int_{\mathbb R^N} |u(x)|^{N} |x|^{-\gamma} dx \geq \|\nabla v\|_N^N + \|v\|_N^N=\|v\|_{W^{1,N}}^N.
\end{equation}
We remark again that \eqref{eq:comparenorm} and \eqref{eq:norm} become equalities if $u$ is radial function for any $\gamma < N$. 
Thus $\|v\|_{W^{1,N}} \leq 1$ if $\|u\|_{X^{1,N}_\gamma} \leq 1$. 
By choosing $G(t) = \Phi_N(\alpha t)$ and $\delta =\beta \geq \gamma$, we get
\begin{multline}
\label{eq:changeintegral*}
	\int_{\mathbb R^N} \Phi_N\left(\alpha|u(x)|^{\frac N{N-1}}\right) |x|^{-\beta} dx\\ 
= \left(\frac {N-\gamma}N\right)^{N-1+ \frac{N(\gamma -\beta)}{N-\gamma}}\int_{\mathbb R^N} \Phi_N\left(\frac N{N-\gamma} \alpha|v(y)|^{\frac N{N-1}}\right) |y|^{-\frac{N(\beta-\gamma)}{N-\gamma}} dy.
\end{multline}
Denote 
\begin{equation*}
\label{btilde}
	\tilde{\beta} = \frac{N(\beta-\gamma)}{N-\gamma} \in [0,N). 
\end{equation*}
By using \eqref{eq:norm} and \eqref{eq:changeintegral*} and applying the singular Trudinger-Moser inequality \eqref{singular TM}, 
we get
\begin{align*}
	&\sup_{u\in X^{1,N}_\gamma(\mathbb R^N), \|u\|_{X^{1,N}_\gamma} \leq 1} \int_{\mathbb R^N} \Phi_N\left(\alpha|u(x)|^{\frac N{N-1}}\right) |x|^{-\beta} dx\\
	&\leq\left(\frac {N-\gamma}N\right)^{N-1+ \frac{N(\gamma -\beta)}{N-\gamma}}\sup_{v\in W^{1,N}(\mathbb R^N), \|v\|_{W^{1,N}} \leq 1}\int_{\mathbb R^N} \Phi_N\left(\frac N{N-\gamma} \alpha|v(y)|^{\frac N{N-1}}\right) |y|^{-\tilde \beta} dy\\
	&=\left(\frac {N-\gamma}N\right)^{N-1+ \frac{N(\gamma -\beta)}{N-\gamma}} \tilde{B}\left(N, \frac{N}{N-\gamma} \alpha, \tilde \beta,0\right)\\
	&<\infty,
\end{align*}
since $\frac{N}{N-\gamma} \alpha \leq \frac{N}{N-\gamma} \alpha_{N,\beta} = \frac{N-\beta}{N-\gamma} \alpha_N = \(\frac{N-\tilde{\beta}}{N}\) \alpha_N = \alpha_{N,\tilde \beta}$.

If $u$ is radial then so is $v$. 
In this case, \eqref{eq:pointwiseineq}, \eqref{eq:comparenorm} become equalities, and hence so does \eqref{eq:norm}. 
Then the conclusion follows again from the singular Trudinger-Moser inequality \eqref{singular TM}.

We finish the proof of Theorem \ref{Theorem:weighted LR} by showing that $\tilde{B}(N, \alpha, \beta, \gamma) = \infty$ and $\tilde{B}_{rad}(N, \alpha, \beta, \gamma) = \infty$
when $\alpha > \alpha_{N, \beta}$. 
Since $\tilde{B}_{rad}(N, \alpha, \beta, \gamma) \leq \tilde{B}(N, \alpha, \beta, \gamma)$, it is enough to prove that $\tilde{B}_{rad}(N, \alpha, \beta, \gamma) = \infty$. 
Suppose the contrary that $\tilde{B}_{rad}(N, \alpha, \beta, \gamma) < \infty$ for some $\alpha > \alpha_{N,\beta}$.
Using again the transformation of functions \eqref{eq:changefunct} for radial functions $u \in X_{\gamma}^{1,N}$, 
we then have equalities in \eqref{eq:pointwiseineq}, \eqref{eq:comparenorm}, and hence in \eqref{eq:norm}. 
Evidently, the transformation of functions \eqref{eq:changefunct} is a bijection between $X^{1,N}_{\gamma,rad}$ and $W^{1,N}_{rad}$ and preserves the equality in \eqref{eq:norm}. 
Consequently, we have
\[
	\tilde{B}_{rad}(N, \alpha, \beta, \gamma) = \left(\frac {N-\gamma}N\right)^{N-1+ \frac{N(\gamma -\beta)}{N-\gamma}} \tilde{B}_{rad}\left(N, \frac{N}{N-\gamma} \alpha, \tilde \beta,0\right),
\]
with $\tilde{\beta} = \frac{N(\beta-\gamma)}{N-\gamma} \in [0,N)$. Hence $\tilde{B}_{rad}\left(N, \frac{N}{N-\gamma} \alpha, \tilde \beta,0\right) < \infty$. 
By rearrangement argument, we have
\[
	\tilde{B}\left(N, \frac{N}{N-\gamma} \alpha, \tilde \beta,0\right) = \tilde{B}_{rad}\left(N, \frac{N}{N-\gamma} \alpha, \tilde \beta,0\right) < \infty
\]
which violates the result of Adimurthi and Yang since $\frac{N}{N-\gamma} \alpha > \alpha_{N,\tilde{\beta}}$.

%
%
For the later purpose, we also prove here directly $\tilde{B}_{rad}(N, \alpha, \beta, \gamma) = \infty$ when $\alpha > \alpha_{N, \beta}$
by using the weighted Moser sequence as in \cite{Ishiwata-Nakamura-Wadade}, \cite{Lam-Lu-Zhang}:
Let $-\infty < \gamma \le \beta < N$ and for $n \in \N$ set
\[
	A_n = \( \frac{1}{\omega_{N-1}} \)^{1/N} \( \frac{n}{N-\beta} \)^{-1/N}, \quad b_n = \frac{n}{N-\beta},
\]
so that $\( A_n b_n \)^{\frac{N}{N-1}} = n/\alpha_{N, \beta}$.
Put
\begin{align}
\label{Moser sequence}
	&u_n = 
	\begin{cases}
	A_n b_n, &\quad \text{if} \, |x| < e^{-b_n}, \\
	A_n \log (1/|x|), &\quad \text{if} \, e^{-b_n} < |x| < 1, \\
	0, &\quad \text{if} \, 1 \le |x|.
	\end{cases}
\end{align}
Then direct calculation shows that
\begin{align}
\label{Moser_estimates(1)}
	&\| \nabla u_n \|_{L^N(\re^N)} = 1, \\
\label{Moser_estimates(2)}
	&\| u_n \|^N_{N, \gamma} = \frac{N-\beta}{(N-\gamma)^{N+1}} \Gamma(N+1) (1/n) + o(1/n)
\end{align}
as $n \to \infty$.
Thus $u_n \in X^{1,N}_{\gamma, rad}(\re^N)$.
In fact for (\ref{Moser_estimates(2)}), we compute
\begin{align*}
	\| u_n \|_{N, \gamma}^N &= \omega_{N-1} \int_0^{e^{-b_n}} (A_n b_n)^N r^{N-1-\gamma} dr + \omega_{N-1} \int_{e^{-b_n}}^1 A_n^N (\log (1/r))^N r^{N-1-\gamma} dr \\
	&= I + II.
\end{align*}
We see
\[
	I = \omega_{N-1} (A_n b_n )^N \left[ \frac{r^{N-\gamma}}{N-\gamma} \right]_{r=0}^{r=e^{-b_n}} = \omega_{N-1} \( \frac{n}{\alpha_{N, \beta}} \)^{N-1} \frac{e^{-(\frac{N-\gamma}{N-\beta}) n} }{N-\gamma} = o(1/n)
\]
as $n \to \infty$.
Also
\begin{align*}
	II &= \( \frac{N-\beta}{n} \) \int_{e^{-b_n}}^1 (\log (1/r))^N r^{N-1-\gamma} dr \\
	&= \( \frac{N-\beta}{n} \) \int_0^{b_n} \rho^N e^{-(N-\gamma) \rho} d\rho
	= \frac{N-\beta}{(N-\gamma)^{N+1}} (1/n) \int_0^{(N-\gamma) b_n} \rho^N e^{-\rho} d\rho \\
	&= \frac{N-\beta}{(N-\gamma)^{N+1}} (1/n) \Gamma(N+1) + o(1/n).
\end{align*}
Thus we obtain (\ref{Moser_estimates(2)}).

Now, put $v_n(x) = \la_n u_n(x)$ where $u_n$ is the weighted Moser sequence in (\ref{Moser sequence}) and
$\la_n > 0$ is chosen so that $\la_n^N + \la_n^N \| u_n \|_{N, \gamma}^N = 1$.
Thus we have $\| \nabla v_n \|_{L^N}^N + \| v_n \|_{N,\gamma}^N = 1$ for any $n \in \N$.
By (\ref{Moser_estimates(2)}) with $\beta = \gamma$, we see that $\la_n^N = 1 - O(1/n)$ as $n \to \infty$.
For $\alpha > \alpha_{N, \beta}$, we calculate
\begin{align*}
	&\int_{\re^N} \Phi_N (\alpha |v_n|^{\frac{N}{N-1}}) \frac{dx}{|x|^{\beta}} 
	\ge \int_{\{ 0 \le |x| \le e^{-b_n} \}} \Phi_N (\alpha |v_n|^{\frac{N}{N-1}}) \frac{dx}{|x|^{\beta}} \\
	&= \int_{\{ 0 \le |x| \le e^{-b_n} \}} \( e^{\alpha |v_n|^{\frac{N}{N-1}}} - \sum_{j=0}^{N-2} \frac{\alpha^j}{j!} |v_n|^{\frac{Nj}{N-1}} \)
	\frac{dx}{|x|^{\beta}} \\
	&\ge \left\{ \exp \( \frac{n \alpha}{\alpha_{N, \beta}} \la_n^{\frac{N}{N-1}} \) - O(n^{N-1}) \right\} 
	\int_{\{ 0 \le |x| \le e^{-b_n} \}} \frac{dx}{|x|^{\beta}} \\
	&\ge \left\{ \exp \( \frac{n \alpha}{\alpha_{N,\beta}} \(1 - O\(\frac{1}{n^{\frac{1}{N-1}}}\)\) \) - O(n^{N-1}) \right\} \( \frac{\omega_{N-1}}{N-\beta} \) e^{-n} \to + \infty
\end{align*}
as $n \to \infty$.
Here we have used that for $0 \le |x| \le e^{-b_n}$, 
\[
	\alpha |v_n|^{\frac{N}{N-1}} = \alpha \la_n^{\frac{N}{N-1}} (A_n b_n)^{\frac{N}{N-1}} = \frac{n \alpha}{\alpha_{N, \beta}} \la_n^{\frac{N}{N-1}}
\]
by definition of $A_n$ and $b_n$. 
Also we used that for $0 \le |x| \le e^{-b_n}$, 
\[
	|v_n|^{\frac{Nj}{N-1}} = \la_n^{\frac{Nj}{N-1}} (A_n b_n)^{\frac{Nj}{N-1}} \le C n^j \le C n^{N-1}
\]
for $0 \le j \le N-2$ and $n$ is large.
This proves Theorem \ref{Theorem:weighted LR} completely.
\qed

%
%

\section{Existence of maximizers for the weighted Trudinger-Moser inequality}

As explained in the Introduction, the existence and non-existence of maximizers for (\ref{singular TM-sup}) is well known. 
Now, let us recall it here.
\begin{prop}
\label{maximizersingular}
The following statements hold,
\begin{enumerate}
\item[(i)] If $N \geq 3$ then $\tilde{B}(N, \alpha,0,0)$ is attained for any $0 < \alpha \leq \alpha_N$ (see \cite{Ishiwata, Li-Ruf}).
\item[(ii)] If $N =2$, there exists $0 < \alpha_*< \alpha_2 =4\pi$ such that $\tilde{B}(2,\alpha,0,0)$ is attained for any $\alpha_* < \alpha \leq \alpha_2$ (see \cite{Ishiwata, Ruf}). 
\item[(iii)] If $\beta \in (0,N)$ and $N \geq 2$ then $\tilde{B}(N,\alpha,\beta,0)$ is attained for any $0< \alpha \leq \alpha_{N,\beta}$ (see \cite{LiYang}).
\item[(iv)] $\tilde{B}(2,\alpha,0,0)$ is not attained for any sufficiently small $\alpha >0$ (see \cite{Ishiwata}).
\end{enumerate}
\end{prop}
The existence part (iii) of Proposition \ref{maximizersingular} is recently proved by X. Li, and Y. Yang \cite{LiYang} by a blow-up analysis.

\begin{remark}
\label{remark1}
By a rearrangement argument, the maximizers for \eqref{singular TM-sup}, if exist, must be a decreasing spherical symmetric function if $\beta \in (0,N)$ and up to a translation if $\beta =0$.
\end{remark}

The proofs of the existence part (i) (ii) of Theorem \ref{Maximizers(radial)} and \ref{Maximizers} are completely similar by using the formula of change of functions \eqref{eq:changefunct} 
and the results on the existence of maximizers for \eqref{singular TM-sup}.
So we prove Theorem \ref{Maximizers} only here.
As we have seen from the proof of Theorem \ref{Theorem:weighted LR} that
\[
	\tilde{B}(N,\alpha,\beta,\gamma) \leq \left(\frac {N-\gamma}N\right)^{N-1+ \frac{N(\gamma -\beta)}{N-\gamma}} \tilde{B}\left(N, \frac{N}{N-\gamma} \alpha, \tilde \beta,0\right)
\]
if $0 \le \gamma \le \beta < N$, where $\tilde{\beta} = N(\beta-\gamma)/(N-\gamma) \in [0,N)$. 
If $N, \alpha, \beta$ and $\gamma$ satisfy the condition (i) and (ii) of Theorem \ref{Maximizers},
then $N$, $N \alpha/(N-\gamma)$ and $\tilde{\beta}$ satisfy the condition (i)--(iii) of Proposition \ref{maximizersingular}, 
hence there exists a maximizer $v \in W^{1,N}(\re^N)$ for $\tilde{B}\( N, \frac{N}{N-\gamma} \alpha, \tilde{\beta} ,0 \)$ with $\|v\|_N^N + \|\nabla v\|_N^N =1$ and
\[
	\int_{\mathbb R^N} \Phi_N\left(\frac N{N-\gamma} \alpha|v(y)|^{\frac N{N-1}}\right) |y|^{-\tilde \beta} dy = \tilde{B}\left(N, \frac{N}{N-\gamma} \alpha, \tilde \beta,0\right).
\]
As mentioned in Remark \ref{remark1}, we can assume that $v$ is a radial function. 
Let $u \in X^{1,N}_{\gamma}$ be a function defined by \eqref{eq:changefunct}. Note that $u$ is also a radial function, hence $\eqref{eq:pointwiseineq}$ becomes an equality. So do \eqref{eq:comparenorm} and \eqref{eq:norm}. Hence, we get
\[
	\|u\|_{X^{1,N}_\gamma}^N = \|\nabla v\|_N^N + \|v\|_N^N =1,
\]
and by \eqref{eq:changeintegral*}
\[
	\int_{\mathbb R^N} \Phi_N\left(\alpha|u(x)|^{\frac N{N-1}}\right) |x|^{-\beta} dx = \left(\frac {N-\gamma}N\right)^{N-1+ \frac{N(\gamma -\beta)}{N-\gamma}}\tilde{B}\left(N, \frac{N}{N-\gamma} \alpha, \tilde \beta,0\right).
\]
This shows that $u$ is a maximizer for $\tilde{B}(N,\alpha,\beta,\gamma)$.
\qed

%
%

\section{Non-existence of maximizers for the weighted Trudinger-Moser inequality}

In this section, we prove the non-existence part (iii) of Theorem \ref{Maximizers}.
The proof of (iii) of Theorem \ref{Maximizers(radial)} is completely similar.
We follow Ishiwata's argument in \cite{Ishiwata}.

Assume $0 \le \beta < 2$, $0 < \alpha \le \alpha_{2, \beta} = 2\pi(2-\beta)$ and recall
\begin{align*}
	&\tilde{B}(2, \alpha, \beta, \beta) 
	= \sup_{u \in X^{1,2}_{\beta}(\re^2) \atop \| u \|_{X^{1,2}_{\beta}(\re^2)} \le 1} \int_{\re^2} \( e^{\alpha u^2} -1 \) \frac{dx}{|x|^{\beta}}.
\end{align*}
We will show that $\tilde{B}(2, \alpha, \beta, \beta)$ is not attained if $\alpha > 0$ sufficiently small.
Set
\begin{align*}
	M = \left\{ u \in X^{1,2}_{\beta}(\re^2) \, : \, \| u \|_{X^{1,2}_{\beta}} =\( \| \nabla u \|_2^2 + \| u \|_{2, \beta}^2 \)^{1/2} = 1 \right\}
\end{align*}
be the unit sphere in the Hilbert space $X^{1,2}_{\beta}(\re^2)$ and
\begin{align*}
	J_{\alpha} :M \to \re, \quad J_{\alpha}(u) = \int_{\re^2} \( e^{\alpha u^2} - 1 \) \frac{dx}{|x|^{\beta}}
\end{align*}
be the corresponding functional defined on $M$.
Actually, we will prove the stronger claim that $J_{\alpha}$ has no critical point on $M$ when $\alpha > 0$ is sufficiently small. 

Assume the contrary that there existed $v \in M$ such that $v$ is a critical point of $J_{\alpha}$ on $M$.
Define an orbit on $M$ through $v$ as
\[
	v_{\tau}(x) = \sqrt{\tau} v(\sqrt{\tau} x) \quad \tau \in (0,\infty), \quad w_{\tau} = \frac{v_{\tau}}{\| v_{\tau} \|_{X^{1,2}_{\beta}}} \in M.
\]
Since $w_{\tau}|_{\tau = 1} = v$, we must have
\begin{equation}
\label{ddt=0}
	\frac{d}{d\tau} \Big|_{\tau = 1} J_{\alpha}(w_{\tau}) = 0.
\end{equation}
Note that
\[
	\| \nabla v_{\tau} \|_{L^2(\re^2)}^2 = \tau  \| \nabla v \|_{L^2(\re^2)}^2, \quad \| v_{\tau} \|_{p, \beta}^p = \tau^{\frac{p+\beta-2}{2}}\| v \|_{p, \beta}^p
\]
for $p > 1$.
Thus,
\begin{align*}
	&J_{\alpha}(w_{\tau}) = \int_{\re^2} \( e^{\alpha w_{\tau}^2} - 1 \) \frac{dx}{|x|^{\beta}} 
= \int_{\re^2} \sum_{j=1}^{\infty} \frac{\alpha^j}{j!} \frac{v_{\tau}^{2j}(x)}{\| v_{\tau} \|_{X^{1,2}_{\beta}}^{2j}} \frac{dx}{|x|^{\beta}} \\
	&= \sum_{j=1}^{\infty} \frac{\alpha^j}{j!} \frac{\| v_{\tau} \|_{2j, \beta}^{2j}}{\(\| \nabla v_{\tau} \|_2^2 + \| v_{\tau}\|_{2, \beta}^2 \)^j}
	= \sum_{j=1}^{\infty} \frac{\alpha^j}{j!}  \frac{\tau^{j -1 +\frac{\beta}{2}} \| v \|_{2j, \beta}^{2j}}{\( \tau \| \nabla v \|_2^2 + \tau^{\frac{\beta}{2}} \| v \|_{2,\beta}^2  \)^j}.
\end{align*}
By using an elementary computation
\begin{align*}
	&f(\tau) = \frac{\tau^{j-1 + \frac{\beta}{2}} c}{(\tau a + \tau^{\frac{\beta}{2}} b)^j}, \quad a = \| \nabla v \|_2^2, \, b = \| v \|_{2,\beta}^2, \, c = \| v \|_{2j,\beta}^{2j}, \\
	&f'(\tau) = (1- \frac{\beta}{2}) \frac{\tau^{j-2 + \frac{\beta}{2}} c}{(\tau a + \tau^{\frac{\beta}{2}} b)^{j+1}} \left\{ -\tau a + (j-1)b \right\},
\end{align*}
we estimate $\frac{d}{d\tau} \Big|_{\tau = 1} J_{\alpha}(w_{\tau})$: 
\begin{align}
\label{d_dtau}
	&\frac{d}{d\tau} \Big|_{\tau = 1} J_{\alpha}(w_{\tau}) \notag \\ 
	&= \sum_{j=1}^{\infty} \left[ \frac{\alpha^j}{j!} (1-\frac{\beta}{2}) \frac{\tau^{j-2 + \beta/2} \| v \|_{2j, \beta}^{2j}}{\( \tau \| \nabla v \|_2^2 + \tau^{\beta/2} \| v \|_{2,\beta}^2 \)^{j+1}} 
	\left\{ -\tau \| \nabla v \|_2^2 + (j-1) \| v \|_{2,\beta}^2 \right\} \right]_{\tau = 1} \notag \\
	&= - \alpha (1-\frac{\beta}{2}) \| \nabla v \|_2^2 \| v \|_{2, \beta}^2 
+ \sum_{j=2}^{\infty} \frac{\alpha^j}{j!} (1-\frac{\beta}{2}) \| v \|_{2j, \beta}^{2j} \left\{ -\| \nabla v \|_2^2 + (j-1) \| v \|_{2,\beta}^2 \right\} \notag \\
	&\le \alpha (1-\frac{\beta}{2}) \| \nabla v \|_2^2 \| v \|_{2,\beta}^2 
	\left\{ -1 + \sum_{j=2}^{\infty} \frac{\alpha^{j-1}}{(j-1)!} \frac{\| v \|_{2j, \beta}^{2j}}{\| \nabla v \|_2^2 \| v \|_{2,\beta}^2}  \right\},
\end{align}
since $-\| \nabla v \|_2^2 + (j-1) \| v \|_{2,\beta}^2 \le j$.

Now, we state a lemma.
Unweighted version of the next lemma is proved in \cite{Ishiwata}:Lemma 3.1,
and the proof of the next is a simple modification of the one given there 
using the weighted Adachi-Tanaka type Trudinger-Moser inequality:
\[
	\tilde{A}(2, \alpha, \beta, \beta) = \sup_{u \in X^{1,2}_{\beta}(\re^2) \setminus \{ 0 \} \atop \| \nabla u \|_{L^2(\re^2)} \le 1} \frac{1}{\| u \|^2_{2, \beta}} \int_{\re^2} \( e^{\alpha u^2} - 1 \) \frac{dx}{|x|^{\beta}} < \infty
\]
for $\alpha \in (0, \alpha_{2, \beta})$ if $\beta \ge 0$, and the expansion of the exponential function.
%
%
\begin{lemma}
\label{Lemma:Ishiwata:Lemma3.1}
For any $\alpha \in (0, \alpha_{2, \beta})$, there exists $C_{\alpha} > 0$ such that
\[
	\| u \|_{2j,\beta}^{2j} \le C_{\alpha} \frac{j!}{\alpha^j}  \| \nabla u \|_2^{2j-2} \| u \|_{2, \beta}^2 
\]
holds for any $u \in X^{1,2}_{\beta}(\re^2)$ and $j \in \N$, $j \ge 2$.
\end{lemma}
By this lemma, 
if we take $\alpha < \tilde{\alpha} < \alpha_{2,\beta}$ and put $C = C_{\tilde{\alpha}}$,
we see
\begin{align*}
	\frac{\| v \|_{2j, \beta}^{2j}}{\| \nabla v \|_2^2 \| v \|_{2,\beta}^2} \le C \frac{j!}{\tilde{\alpha}^j} \| \nabla v \|_{2j}^{2j-4}
\le C \frac{j!}{\tilde{\alpha}^j}
\end{align*}
for $j \ge 2$ since $v \in M$.
Thus we have
\[
	 \sum_{j=2}^{\infty} \frac{\alpha^{j-1}}{(j-1)!} \frac{\| v \|_{2j, \beta}^{2j}}{\| \nabla v \|_2^2 \| v \|_{2,\beta}^2} \le \sum_{j=2}^{\infty} \frac{C \alpha^{j-1}}{(j-1)!} \frac{j!}{\tilde{\alpha}^j} 
	= (\frac{C \alpha}{\tilde{\alpha}^2}) \sum_{j=2}^{\infty} \( \frac{\alpha}{\tilde{\alpha}} \)^{j-2} j \le \alpha C^{\prime}
\]
for some $C^{\prime} > 0$.
Inserting this into the former estimate (\ref{d_dtau}), we obtain
\begin{align*}
	\frac{d}{d\tau} \Big|_{\tau = 1} J_{\alpha}(w_{\tau}) \le (1-\frac{\beta}{2}) \alpha \| \nabla v \|_2^2 \| v \|_{2,\beta}^2 (-1 + C^{\prime}\alpha) < 0
\end{align*}
when $\alpha >0$ is sufficiently small.
This contradicts to (\ref{ddt=0}).
\qed

%
%

\section{Proof of Theorem \ref{Theorem:relation} and \ref{Theorem:asymptotic}.}

In this section, we prove Theorem \ref{Theorem:relation} and Theorem \ref{Theorem:asymptotic}.
As stated in the Introduction, we follow the argument by Lam-Lu-Zhang \cite{Lam-Lu-Zhang}.
First, we prepare several lemmata.

%
%

\begin{lemma}
\label{Lemma1}
Assume (\ref{assumption:weighted AT}) and set
\begin{equation}
\label{A_hat}
	\widehat{A}(N, \alpha, \beta, \gamma) 
	= \sup_{{u \in X^{1,N}_{\gamma}(\re^N) \setminus \{ 0 \} \atop \| \nabla u \|_{L^N(\re^N)} \le 1} \atop \| u \|_{N, \gamma} = 1} 
	\int_{\re^N} \Phi_N (\alpha |u|^{\frac{N}{N-1}}) \frac{dx}{|x|^{\beta}}.
\end{equation}
Let $\tilde{A}(N, \alpha, \beta, \gamma)$ be defined as in (\ref{weighted AT-sup}).
Then $\tilde{A}(N, \alpha, \beta, \gamma) = \widehat{A}(N, \alpha, \beta, \gamma)$ for any $\alpha > 0$. 
Similarly, $\tilde{A}_{rad}(N, \alpha, \beta, \gamma) = \widehat{A}_{rad}(N, \alpha, \beta, \gamma)$ for any $\alpha > 0$,
where $\widehat{A}_{rad}(N, \alpha, \beta, \gamma)$ is defined similar to (\ref{A_hat}) and $\widehat{A}_{rad}(N, \alpha, \beta, \gamma)$ is defined in (\ref{weighted AT-sup(radial)}).
\end{lemma}

\begin{proof}
For any	$u \in X^{1,N}_{\gamma}(\re^N) \setminus \{ 0 \}$ and $\la >0$, we put $u_{\la}(x) = u(\la x)$ for $x \in \re^N$.
Then it is easy to see that
\begin{equation}
\label{scaling}
	\begin{cases}
	&\| \nabla u_{\la} \|_{L^N(\re^N)}^N = \| \nabla u \|_{L^N(\re^N)}^N, \\ 
	&\| u_{\la} \|^N_{N, \gamma} = \la^{-(N-\gamma)} \| u \|^N_{N, \gamma}.
	\end{cases}
\end{equation}
Thus for any $u \in X^{1,N}_{\gamma}(\re^N) \setminus \{ 0 \}$ with $\| \nabla u \|_{L^N(\re^N)} \le 1$,
if we choose $\la = \| u \|^{N/(N-\gamma)}_{N, \gamma}$, then $u_{\la} \in X^{1,N}_{\gamma}(\re^N)$ satisfies 
\[
	\| \nabla u_{\la} \|_{L^N(\re^N)} \le 1  \quad \text{and} \quad \| u_{\la} \|^N_{N, \gamma} = 1.
\]
Thus 
\[
	\widehat{A}(N, \alpha, \beta, \gamma) \ge
	\int_{\re^N} \Phi_N (\alpha |u_{\la}|^{\frac{N}{N-1}}) \frac{dx}{|x|^{\beta}} =
	\frac{1}{\| u \|_{N, \gamma}^{\frac{N(N-\beta)}{N-\gamma}}} \int_{\re^N} \Phi_N (\alpha |u|^{\frac{N}{N-1}}) \frac{dx}{|x|^{\beta}} 
\] 
which implies $\widehat{A}(N, \alpha, \beta, \gamma) \ge \tilde{A}(N, \alpha, \beta, \gamma)$.
The opposite inequality is trivial.
\end{proof}

%
%
\begin{lemma}
\label{Lemma2}
Assume (\ref{assumption:weighted AT}) and set $\tilde{B}(N, \beta, \gamma)$ as in (\ref{weighted LR-sup-critical}).
Then we have
\[
	 \tilde{A}(N, \alpha, \beta, \gamma) \le \( \frac{\(\frac{\alpha}{\alpha_{N,\beta}}\)^{N-1}}{1 - \(\frac{\alpha}{\alpha_{N,\beta}}\)^{N-1}} \)^{\frac{N-\beta}{N-\gamma}} \tilde{B}(N, \beta, \gamma)
\]
for any $0 < \alpha < \alpha_{N, \beta}$.
The same relation holds for $\tilde{A}_{rad}(N, \alpha, \beta, \gamma)$ in (\ref{weighted AT-sup(radial)}) and $\tilde{B}_{rad}(N, \beta, \gamma)$ in (\ref{weighted LR-sup-critical(radial)}).
\end{lemma}

\begin{proof}
Choose any $u \in X^{1,N}_{\gamma}$ with $\| \nabla u \|_{L^N(\re^N)} \le 1$ and  $\| u \|_{N, \gamma} = 1$. 
Put $v(x) = C u(\la x)$ where $C \in (0,1)$ and $\la > 0$ are defined as 
\[
	C = \( \frac{\alpha}{\alpha_{N, \beta}} \)^{\frac{N-1}{N}} \quad \text{and} \quad
	\la = \( \frac{C^N}{1 - C^N} \)^{1/(N-\gamma)}.
\]
Then by scaling rules (\ref{scaling}), we see
\begin{align*}
	\| v \|^N_{X^{1,N}_{\gamma}} &= \| \nabla v \|^N_N + \| v \|^N_{N, \gamma}
	= C^N \| \nabla u \|^N_N + \la^{-(N-\gamma)} C^N \| u \|^N_{N, \gamma} \\
	&\le C^N + \la^{-(N-\gamma)} C^N = 1.
\end{align*}
Also we have
\begin{align*}
	\int_{\re^N} \Phi_N (\alpha_{N, \beta} |v|^{\frac{N}{N-1}}) \frac{dx}{|x|^{\beta}} &=
	\la^{-(N-\beta)} \int_{\re^N} \Phi_N \( \alpha_{N, \beta} C^{\frac{N}{N-1}} |u|^{\frac{N}{N-1}} \) \frac{dx}{|x|^{\beta}} \\
	&= \la^{-(N-\beta)} \int_{\re^N} \Phi_N \( \alpha |u|^{\frac{N}{N-1}} \) \frac{dx}{|x|^{\beta}}.
\end{align*}
Thus testing $\tilde{B}(N, \beta, \gamma)$ by $v$, we see
\[
	\tilde{B}(N, \beta, \gamma) \ge \( \frac{1 - C^N}{C^N} \)^{\frac{N-\beta}{N-\gamma}} \int_{\re^N} \Phi_N \( \alpha |u|^{\frac{N}{N-1}} \) \frac{dx}{|x|^{\beta}}.
\]
By taking the supremum for $u \in X^{1,N}_{\gamma}$ with $\| \nabla u \|_{L^N(\re^N)} \le 1$ and  $\| u \|_{N, \gamma} = 1$,
we have
\[
	 \tilde{B}(N, \beta, \gamma) \ge \( \frac{1 - \(\frac{\alpha}{\alpha_{N,\beta}}\)^{N-1}}{\(\frac{\alpha}{\alpha_{N,\beta}}\)^{N-1}} \)^{\frac{N-\beta}{N-\gamma}} \widehat{A}(N, \alpha, \beta, \gamma).
\]
Finally, Lemma \ref{Lemma1} implies the result.
The proof of
\[
	 \tilde{B}_{rad}(N, \beta, \gamma) \ge \( \frac{1 - \(\frac{\alpha}{\alpha_{N,\beta}}\)^{N-1}}{\(\frac{\alpha}{\alpha_{N,\beta}}\)^{N-1}} \)^{\frac{N-\beta}{N-\gamma}} \widehat{A}_{rad}(N, \alpha, \beta, \gamma)
\]
is similar.
\end{proof}

%
%

\vspace{1em}\noindent
{\it Proof of Theorem \ref{Theorem:relation}}:
We prove the relation between $\tilde{B}(N, \beta, \gamma)$ and $\tilde{A}(N, \alpha, \beta, \gamma)$ only.
The assertion that 
\[
	\tilde{B}(N, \beta, \gamma) \ge \sup_{\alpha \in (0, \alpha_{N, \beta})} \( \frac{1 - \(\frac{\alpha}{\alpha_{N,\beta}}\)^{N-1}}{\(\frac{\alpha}{\alpha_{N,\beta}}\)^{N-1}} \)^{\frac{N-\beta}{N-\gamma}} 
\tilde{A}(N, \alpha, \beta, \gamma)
\]
follows from Lemma \ref{Lemma2}. 
Note that $\tilde{B}(N, \beta,\gamma) < \infty$ when $0 \le \gamma \le \beta < N$ by Theorem \ref{Theorem:weighted LR}.

Let us prove the opposite inequality.
Let $\{ u_n \} \subset X^{1,N}_{\gamma}(\re^N)$, $u_n \ne 0$, $\| \nabla u_n \|^N_{L^N} + \| u_n \|^N_{N, \gamma} \le 1$,
be a maximizing sequence of $\tilde{B}(N, \beta, \gamma)$:
\[
	\int_{\re^N} \Phi_N (\alpha_{N, \beta} |u_n|^{\frac{N}{N-1}}) \frac{dx}{|x|^{\beta}} = \tilde{B}(N, \beta, \gamma) + o(1)
\]
as $n \to \infty$.
We may assume $\| \nabla u_n \|^N_{L^N(\re^N)} < 1$ for any $n \in \N$.
Define
\[
	\begin{cases}
	&v_n(x) = \frac{u_n(\la_n x)}{\| \nabla u_n \|_N}, \quad (x \in \re^N) \\
	&\la_n = \( \frac{1 - \| \nabla u_n \|_N^N}{\| \nabla u_n \|^N_N} \)^{1/(N-\gamma)} > 0.
	\end{cases}
\]
Thus by (\ref{scaling}), we see
\begin{align*}
	&\| \nabla v_n \|^N_{L^N(\re^N)} = 1, \\
	&\|  v_n \|^{\frac{N(N-\beta)}{N-\gamma}}_{N, \gamma} = \( \frac{\la_n^{-(N-\gamma)}}{\| \nabla u_n \|^N_N} \| u_n \|^N_{N, \gamma} \)^{\frac{N-\beta}{N-\gamma}}
	= \( \frac {\| u_n \|^N_{N, \gamma}}{1 - \| \nabla u_n \|^N_N} \)^{\frac{N-\beta}{N-\gamma}} \le 1,
\end{align*}
since $\| \nabla u_n \|^N_N + \| u_n \|^N_{N, \gamma} \le 1$.
Thus, setting  
\[
	\alpha_n = \alpha_{N, \beta} \| \nabla u_n \|^{\frac{N}{N-1}}_N < \alpha_{N, \beta}
\]
for any $n \in \N$,
we may test $\tilde{A}(N, \alpha_n, \beta, \gamma)$ by $\{ v_n \}$, which results in
\begin{align*}
	\tilde{B}(N, \beta, \gamma) + o(1) &= \int_{\re^N} \Phi_N (\alpha_{N, \beta} |u_n(y)|^{\frac{N}{N-1}}) \frac{dy}{|y|^{\beta}} \\ 
	&= \la_n^{N-\beta} \int_{\re^N} \Phi_N (\alpha_{N, \beta} \| \nabla u_n \|_N^{\frac{N}{N-1}} |v_n(x)|^{\frac{N}{N-1}}) \frac{dx}{|x|^{\beta}} \\ 
	&= \la_n^{N-\beta} \int_{\re^N} \Phi_N (\alpha_n |v_n(x)|^{\frac{N}{N-1}}) \frac{dx}{|x|^{\beta}} \\ 
	&\le \la_n^{N-\beta} \( \frac{1}{\| v_n \|_{N, \beta}^N} \)^{\frac{N-\beta}{N-\gamma}} \int_{\re^N} \Phi_N (\alpha_n |v_n(x)|^{\frac{N}{N-1}}) \frac{dx}{|x|^{\beta}} \\ 
	&\le \la_n^{N-\beta} \tilde{A}(N, \alpha_n, \beta, \gamma) = \( \frac{1 - \| \nabla u_n \|_{N}^N}{\| \nabla u_n \|^N_{N}} \)^{\frac{N-\beta}{N-\gamma}} \tilde{A}(N, \alpha_n, \beta, \gamma) \\
	&= \( \frac{1 - \(\frac{\alpha_n}{\alpha_{N,\beta}}\)^{N-1}}{\(\frac{\alpha_n}{\alpha_{N,\beta}}\)^{N-1}} \)^{\frac{N-\beta}{N-\gamma}} \tilde{A}(N, \alpha_n, \beta, \gamma) \\
	&\le \sup_{\alpha \in (0, \alpha_{N, \beta})} \( \frac{1 - \(\frac{\alpha}{\alpha_{N,\beta}}\)^{N-1}}{\(\frac{\alpha}{\alpha_{N,\beta}}\)^{N-1}} \)^{\frac{N-\beta}{N-\gamma}} \tilde{A}(N, \alpha, \beta, \gamma).
\end{align*}
Here we have used a change of variables $y = \la_n x$ for the second equality, and $\| v_n \|^{\frac{N(N-\beta)}{N-\gamma}}_{N, \gamma} \le 1$ for the first inequality.
Letting $n \to \infty$, we have the desired result.
\qed

%
%

\vspace{1em}\noindent
{\it Proof of Theorem \ref{Theorem:asymptotic}}: 
Again, we prove theorem for $\tilde{A}(N, \alpha, \beta, \gamma)$ only.
The assertion that
\[
	\tilde{A}(N, \alpha, \beta, \gamma) \le \( \frac{C_2}{1 - \(\frac{\alpha}{\alpha_{N,\beta}}\)^{N-1}} \)^{\frac{N-\beta}{N-\gamma}}
\]
follows form Theorem \ref{Theorem:relation} and the fact that $\tilde{B}(N, \beta, \gamma) < \infty$ when $0 \le \gamma \le \beta < N$.

For the rest, we need to prove that there exists $C > 0$ such that for any $\alpha < \alpha_{N, \beta}$ sufficiently close to $\alpha_{N, \beta}$,
it holds that
\begin{equation}
\label{Lower}
	\( \frac{C}{1 - \(\frac{\alpha}{\alpha_{N,\beta}}\)^{N-1}} \)^{\frac{N-\beta}{N-\gamma}} \le \tilde{A}(N, \alpha, \beta, \gamma).
\end{equation}
For that purpose, we use the weighted Moser sequence (\ref{Moser sequence}) again.
By (\ref{Moser_estimates(2)}), we have $N_1 \in \N$ such that
if $n \in \N$ satisfies $n \ge N_1$, then it holds
\begin{equation}
\label{lower1}
	\| u_n \|_{N, \gamma}^N \le \frac{2 (N-\gamma) \Gamma (N+1)}{(N-\beta)^{N+1}} (1/n). 
\end{equation}
On the other hand,
\begin{align*}
	\int_{\re^N} \Phi_N(\alpha |u_n|^{N/(N-1)}) \frac{dx}{|x|^{\beta}} &\ge \omega_{N-1} \int_0^{e^{-b_n}} \Phi_N \( \alpha (A_n b_n)^{N/(N-1)} \) r^{N-1-\beta} dr \\
	&= \frac{\omega_{N-1}}{N-\beta} \Phi_N \( (\alpha/\alpha_{N,\beta}) n \) \left[ r^{N-\beta} \right]_{r=0}^{r = e^{-b_n}} \\
	& = \frac{\omega_{N-1}}{N-\beta} \Phi_N \( (\alpha/\alpha_{N,\beta}) n \) e^{-n}.
\end{align*}
Note that there exists $N_2 \in \N$ such that if $n \ge N_2$ then $\Phi_N \( (\alpha/\alpha_{N,\beta}) n \) \ge \frac{1}{2} e^{(\alpha/\alpha_{N,\beta}) n}$.
Thus we have
\begin{equation}
\label{lower2}
	\int_{\re^N} \Phi_N(\alpha |u_n|^{N/(N-1)}) \frac{dx}{|x|^{\beta}} \ge \frac{1}{2} \( \frac{\omega_{N-1}}{N-\beta} \) e^{-(1 - \frac{\alpha}{\alpha_{N,\beta}}) n}.
\end{equation}
Combining (\ref{lower1}) and (\ref{lower2}), we have $C_1(N, \beta, \gamma) > 0$ such that
\begin{equation}
\label{lower3}
	\frac{1}{ \| u_n \|_{N, \gamma}^{\frac{N(N-\beta)}{N-\gamma}}} \int_{\re^N} \Phi_N (\alpha |u_n|^{N/(N-1)}) \frac{dx}{|x|^{\beta}} \ge C_1(N, \beta, \gamma) n^{\frac{N-\beta}{N-\gamma}} e^{-(1 - \frac{\alpha}{\alpha_{N,\beta}}) n}
\end{equation}
holds when $n \ge \max \{ N_1, N_2 \}$.

Note that $\lim_{x \to 1} \( \frac{1-x^{N-1}}{1-x} \) = N-1$, thus 
\[
	\frac{1 - (\alpha/\alpha_{N, \beta})^{N-1}}{1 - (\alpha/\alpha_{N,\beta})} \ge \frac{N-1}{2}
\]
if $\alpha/\alpha_{N, \beta} < 1$ is very close to $1$.
Now, for any $\alpha > 0$ sufficiently close to $\alpha_{N, \beta}$ so that 
\begin{align}
\label{cond_alpha}
	\begin{cases}
	&\max \{ N_1, N_2 \} < \( \frac{2}{1 - \alpha/\alpha_{N,\beta}} \), \\
	&\frac{1 - (\alpha/\alpha_{N, \beta})^{N-1}}{1 - (\alpha/\alpha_{N,\beta})} \ge \frac{N-1}{2},
	\end{cases}
\end{align}
we can find $n \in \N$ such that
\begin{align}
\label{cond_n}
	\begin{cases}
	&\max \{ N_1, N_2 \} \le n \le \( \frac{2}{1 - \alpha/\alpha_{N,\beta}} \), \\
	&\( \frac{1}{1 - \alpha/\alpha_{N,\beta}} \) \le n.
	\end{cases}
\end{align}
We fix $n \in \N$ satisfying (\ref{cond_n}).
Then by $1 \le n (1 - \alpha/\alpha_{N, \beta}) \le 2$, (\ref{lower3}) and (\ref{cond_alpha}), we have
\begin{align*}
	&\frac{1}{ \| u_n \|_{N, \beta}^N} \int_{\re^N} \Phi_N(\alpha |u_n|^{N/(N-1)}) \frac{dx}{|x|^{\beta}} \ge C_1(N, \beta, \gamma) n^{\frac{N-\beta}{N-\gamma}} e^{-2} \\
	&\ge C_2(N, \beta, \gamma) \( \frac{1}{1 - (\alpha/\alpha_{N, \beta})} \)^{\frac{N-\beta}{N-\gamma}} \ge \frac{N-1}{2} C_2(N, \beta, \gamma) \( \frac{1}{1 - (\alpha/\alpha_{N, \beta})^{N-1}} \)^{\frac{N-\beta}{N-\gamma}} \\
	&= C_3(N, \beta, \gamma) \( \frac{1}{1 - (\alpha/\alpha_{N, \beta})^{N-1}} \)^{\frac{N-\beta}{N-\gamma}},
\end{align*}
where $C_2(N, \beta, \gamma) = e^{-2} C_1(N, \beta, \gamma)$ and $C_3(N, \beta, \gamma) = \frac{N-1}{2} C_2(N, \beta, \gamma)$.
Thus we have (\ref{Lower})
for some $C > 0$ independent of $\alpha$ which is sufficiently close to $\alpha_{N, \beta}$.
\qed

%
%

\vspace{1em}\noindent
{\bf Acknowledgments.}

The first author (V. H. N.) was supported by CIMI's postdoctoral research fellowship. The second author (F.T.) was supported by 
JSPS Grant-in-Aid for Scientific Research (B), No.15H03631, and
JSPS Grant-in-Aid for Challenging Exploratory Research, No.26610030.

\end{document}